\documentclass[12pt]{amsart}
\usepackage{fullpage}
\usepackage{amsmath, amsthm, amssymb}
\usepackage[normalem]{ulem}
\usepackage[dvipsnames]{xcolor}
\usepackage{epsfig}
\usepackage{enumitem,subcaption}
\setenumerate{label=\arabic*.,ref=\arabic*}
\usepackage{setspace}
\usepackage{comment}
\usepackage[all]{xy}
\usepackage{hyperref}
\usepackage{setspace}
\usepackage[colorinlistoftodos]{todonotes}

\theoremstyle{plain}
\newtheorem{thm}{Theorem}[section]

\newtheorem{cor}[thm]{Corollary}

\newtheorem{prop}[thm]{Proposition}

\theoremstyle{definition}
\newtheorem{defn}[thm]{Definition}
\newtheorem{definition}[thm]{Definition}
\newtheorem{remark}[thm]{Remark}
\newtheorem{ex}[thm]{Example}
\newtheorem{nota}[thm]{Notation}
\newtheorem{notarem}[thm]{Notation and Remarks}

\DeclareMathOperator{\tr}{tr}
\definecolor{dgreen}{RGB}{0,128,0}

\newcommand{\orb}{\mathcal O}

\newcommand{\R}{\mathbf R}
\newcommand{\Z}{\mathbf Z}

\newcommand{\bs}{\backslash}

\newcommand{\vol}{{\rm vol}}

\newcommand{\iso}{\iiso}
\newcommand{\supp}{{\rm supp}}

\newcommand{\restr}[1]{\lower0.4ex\hbox{$|$}\lower0.7ex
	\hbox{$\scriptstyle{#1}$}}
\definecolor{olive}{RGB}{128,128,0}

\definecolor{orange}{rgb}{1,0.5,0}

\definecolor{darkpastelpurple}{rgb}{0.59,0.44,0.84}

\newcommand{\cc}[1][U]{(\widetilde{#1},G_{#1}, \pi_{#1})}
\newcommand{\wtu}{\widetilde{U}}

\newcommand{\iiso}{\operatorname{Iso}}
\newcommand{\isomax}{\iiso^{\rm max}}

\newcommand{\coloneqq}{:=}

\newcommand{\End}{\operatorname{End}}

  \def\R{\mathbb R} \def\Z{\mathbb{Z}}

\def\O{{\mathcal{O}}}
\def\a{{\alpha}}

\def\g{{\gamma}}
\def\ga{{G_\a}}
\def\hm{{L^{(m)}}}
\def\tham{{\widetilde{H}^{(m)}_\a}}
\def\ham{{H^{(m)}_\a}}
\def\tha{{\widetilde{H}_\a}}
\def\ha{{H_\a}}

\def\psia{{\psi_\a}}
\def\too{{t\to 0^+}}

\def\ua{{U_\a}}
\def\va{{V_\a}}
\def\wa{{W_\a}}

\def\tu{{\widetilde{U}}}
\def\tul{{\widetilde{u}}}
\def\tua{{\widetilde{U}_\a}}
\def\tva{{\widetilde{V}_\a}}
\def\twa{{\widetilde{W}_\a}}

\def\tpsia{{\widetilde{\psi}_\a}}

\def\tx{{\widetilde{x}}}
\def\ty{{\widetilde{y}}}

\def\heat{{\left(\frac{\partial}{\partial t}+\Delta_x\right)}}
\def\heatp{{\left(\partial_t+\Delta^p_x\right)}}

\def\bs{{\backslash}}

\def\shag{{(4\pi t)^{-d/2}e^{-d(\tx,\g(\ty))^2/4t}}}

\def\sha{{(4\pi t)^{-d/2}e^{-d(\tx,\ty)^2/4t}}}
\def\shapl{{(4\pi t)^{-d/2}e^{-d^2(x,y)/4t}}}

\def\ehapl{{(u^p_0(x,y)+\dots +t^mu^p_m(x,y))}}
\def\ehap{{(u^p_0(\tx,\ty)+\dots +t^mu^p_m(\tx,\ty))}}

\def\ehagi{{(\g_{.,2}^*u^p_0(\tx,\ty)+ t\g_{.,2}^*u^p_1(\tx,\ty)+\dots)}}

\def\gn{{\Iso(N)}}
\def\e{{\epsilon}}

\def\rp{{\R_+}}

\def\tetaa{{\widetilde{\eta}_\a}}
\def\etaa{{\eta_\a}}

\def\Iso{{\operatorname{{Iso}}}}

\def\Fix{{\operatorname{{Fix}}}}
\def\x{{\rm{x}}}
\def\y{{\rm y}}
\def\z{{\rm z}}
\newcommand{\codim}{\operatorname{codim}}
\def\Id{{\operatorname{{Id}}}}

\makeatletter
\@namedef{subjclassname@2020}{\textup{2020} Mathematics Subject Classification}
\makeatother

\begin{document}

\title{Do the Hodge spectra distinguish orbifolds from manifolds?\\ Part 1}

\author{Katie Gittins}
\address{Department of Mathematical Sciences, Durham University,
Mathematical Sciences \& Computer Science Building,
Upper Mountjoy Campus, Stockton Road,
Durham DH1 3LE,
United Kingdom.}
\email{katie.gittins@durham.ac.uk}

\author{Carolyn Gordon}
\address{Department of Mathematics, Dartmouth College, Hanover, NH, 03755, USA.}
\email{csgordon@dartmouth.edu}

\author{Magda Khalile}
\address{Institut f\"{u}r Analysis, Leibniz Universit\"{a}t Hannover, Welfengarten 1, 30167, Hannover, Germany.}
\email{magda.khalile@math.uni-hannover.de}

\author{Ingrid Membrillo Solis}
\address{Mathematical Sciences, University of Southampton, University Road,  Southampton SO17~1BJ, United Kingdom.}
\email{i.membrillo-solis@soton.ac.uk}

\author{Mary Sandoval}
\address{Department of Mathematics, Trinity College, 300 Summit Street, Hartford, CT, 06106, USA.}
\email{mary.sandoval@trincoll.edu}

\author{Elizabeth Stanhope}
\address{Department of Mathematical Sciences, Lewis \& Clark College, Portland, OR, 97219, USA.}
\email{stanhope@lclark.edu}


\begin{abstract}
We examine the relationship between the singular set of a compact Riemannian orbifold and the spectrum of the Hodge Laplacian on $p$-forms by computing the heat invariants associated to the $p$-spectrum. We show that the heat invariants of the $0$-spectrum together with those of the $1$-spectrum for the corresponding Hodge Laplacians are sufficient to distinguish orbifolds with singularities from manifolds as long as the singular sets have codimension $\le 3.$ This is enough to distinguish orbifolds from manifolds for dimension $\le 3.$ 
\end{abstract}

\maketitle
\tableofcontents

\setcounter{section}{1}
\section*{Introduction}

A (Riemannian) orbifold  is  a  versatile generalization  of  a (Riemannian) manifold that  permits  the  presence of well-structured  singular  points. Orbifolds appear in a variety of mathematical areas and have applications in physics, in particular to string theory.

Orbifolds are locally the orbit spaces of effective finite group actions on $\R^d$.  Orbifold singularities correspond to orbits with non-trivial isotropy and thus each singularity has an associated ``isotropy type.''   The notions of the Laplace-Beltrami operator and Hodge Laplacian extend to Riemannian orbifolds and many results on the spectral geometry of Riemannian manifolds, such as Weyl asymptotics and Cheeger's inequality, extend to this more general setting.    

Because  the possible presence of singular  points is a defining  characteristic  of  the class of  orbifolds,  the question ``Can one hear the singularities of an orbifold?" is  fundamental  in  the  study  of  the  spectral  geometry  of  orbifolds.    More precisely, one asks:

\begin{enumerate}
\item \label{it:orbimani1} Do spectral data distinguish orbifolds with singularities from manifolds?
\item  \label{it:orbimani2} Do spectral data detect the topology and geometry of the set of singular points including the isotropy types of singularities?
\end{enumerate}
We will always assume all orbifolds under consideration are connected and compact without boundary.

Many authors have addressed these questions for the spectrum of the Laplace-Beltrami operator.   However, the question of whether the spectrum of the Laplace-Beltrami operator \emph{always} distinguishes Riemannian orbifolds with singularities from Riemannian manifolds, remains open.    In this article, we focus primarily on question \ref{it:orbimani1} and consider whether additional spectral data, specifically the spectra of the Hodge Laplacians on $p$-forms, suffice to detect the presence of singularities.    We will say that two closed Riemannian orbifolds are $p$-\emph{isospectral} if their Hodge Laplacians acting on $p$-forms are isospectral.  In particular, 0-isospectrality means that the Laplace-Beltrami operators are isospectral.

We construct the fundamental solution of the heat equation and the heat trace for the $p$-spectrum on closed Riemannian orbifolds, analogous to that of the 0-spectrum in \cite{DGGW08} discussed below, and then apply the resulting spectral invariants to question \ref{it:orbimani1}.   Our main result is that the 0-spectrum and 1-spectrum together distinguish orbifolds with a sufficiently large singular set from manifolds.  Orbifolds are stratified spaces in which the regular points form an open dense stratum, and each lower-dimensional stratum is a connected component of the set of singularities of a given isotropy type. By the \emph{codimension} of the singular set, we mean the minimum codimension of the singular strata.

 \begin{thm}\label{thm:main1V4}  The $0$-spectrum and $1$-spectrum together distinguish closed Riemannian orbifolds with singular sets of codimension $\leq 3$  from closed Riemannian manifolds.
 \end{thm}

Consequently:

\begin{cor}\label{thm:lowdim} For $d\leq 3$, the $0$-spectrum and $1$-spectrum together distinguish singular closed $d$-dimensional Riemannian orbifolds from closed Riemannian manifolds.
\end{cor}

 The case $d=3$ is perhaps of particular interest since 3-dimensional orbifolds play a role in Thurston's Geometrization Conjecture, later proven by Grigori Perel'man.

In a forthcoming paper (Part 2), joint with Juan Pablo Rossetti, we address both questions \ref{it:orbimani1} and \ref{it:orbimani2} for the individual $p$-spectra. For example, we give conditions on the codimension of the singular set which guarantee that the volume of the singular set is determined by spectral data, and in many cases we show by providing counterexamples that the conditions are sharp.

\subsection{Partial review of known results for the Laplace-Beltrami operator.}

The earliest results in the spectral geometry of Riemannian orbifolds are those of Yuan-Jen Chiang \cite{YJC} who established the existence of the Laplace spectrum and heat kernel of a Riemannian orbifold, and Carla Farsi \cite{F01} who extended the Weyl law to the orbifold setting.  Results from a variety of authors soon followed, including several examples of isospectral non-isometric orbifolds with a variety of properties.   For example, Naveed Shams, the last author, and David Webb \cite{BSW06} produced arbitrarily large finite families of mutually strongly isospectral orbifolds with non-isomorphic maximal isotropy groups of the same order.  In particular, this implies that they are $p$-isospectral for all $p$.   Shortly afterward, J. P. Rossetti, Dorothee Schueth, and Martin Weilandt \cite{RSW08} produced examples of strongly isospectral orbifolds with maximal isotropy groups of distinct orders.

Emily Dryden and Alexander Strohmaier \cite{DS09} proved that hyperbolic orientable closed 2-orbifolds are 0-isospectral if and only if they have the same geodesic length spectrum and the same number of cone points of each order.  Benjamin Linowitz and Jeffrey Meyer \cite{LM17} obtained interesting results for locally symmetric spaces.  A series of papers extended to the orbifold setting the result of Robert Brooks, Peter Perry, and Peter Petersen \cite{BPP92} that a set $\mathcal I_k(d)$ of 0-isospectral $d$-manifolds, sharing a uniform lower bound $k$ on sectional curvature, contains manifolds of only finitely many homeomorphism types (diffeomorphism types for $d \neq 4$).  Specifically, the last author \cite{St05} obtained an upper bound on the order of isotropy groups that can arise in elements of the set $\mathcal I_k(d)$ (expanded to contain orbifolds), Emily Proctor and the last author \cite{PrSt10} showed orbifold diffeomorphism finiteness for $\mathcal I_k(2)$, Proctor \cite{Pr12} obtained orbifold homeomorphism finiteness for $\mathcal I_k(d)$ assuming only isolated singularities, and John Harvey \cite{H16} obtained orbifold homeomorphism finiteness for the full set $\mathcal I_k(d)$. In addition, Farsi, Proctor and Christopher Seaton \cite{FPS14} introduced and studied a stronger notion of isospectrality of orbifolds.

Recall that the trace of the heat kernel of the Laplace-Beltrami operator on a closed $d$-dimensional Riemannian manifold ${M}$ is a primary source of spectral invariants.  Denoting the spectrum by $\{\lambda_j\}_{j=1}^\infty$, the heat trace yields an asymptotic expansion of the following form as $t \downarrow 0^{+}$:
\begin{equation*}
    \sum_{j=1}^\infty e^{-\lambda_j t}
    \sim
    (4\pi t)^{-d/2} \sum_{k=0}^{\infty} a_k({M}) \, t^k.
\end{equation*}
The coefficients $a_k$ are spectral invariants referred to as the \emph{heat invariants}. The first two, for example, give the volume and the total scalar curvature of the Riemannian manifold.

A orbifold is said to be \emph{good} if it is a global quotient $\O=\Gamma\bs M$ of a manifold $M$ by a (possibly infinite) discrete group $\Gamma$ acting effectively and properly by diffeomorphisms.  (We say $M$ is a cover of $\O$ even though the group action need not be free.)   In \cite{D76}, Harold Donnelly obtained an asymptotic expansion of the form
$$
    \sum_{j=1}^\infty e^{-\lambda_j t}
    \sim
(4\pi t)^{-d/2} \sum_{k=0}^{\infty} c_k(\mathcal{\O}) \, t^{k/2}$$ for the heat trace of a good Riemannian orbifold by adding contributions from each element $\gamma\in \Gamma$.  Each $c_k(\O)$ is a sum of terms $c_k(\O,\gamma)$, $\gamma\in \Gamma$, which in turn is a sum of integrals over the connected components of the fixed point set of $\gamma$.  The integrands are universal expressions in the germs of the metric and the isometry $\gamma$. 

By applying 
Donnelly's heat trace asymptotics for good orbifolds, Craig Sutton \cite{S10} showed that if a closed Riemannian orbifold $\O$ with singularities and a closed Riemannian manifold $M$ have 0-isospectral finite Riemannian covers then they cannot be $0$-isospectral. In particular if $\O$ and $M$ have a common finite Riemannian cover, they cannot be $0$-isospectral.  Juan Pablo Rossetti and the second author showed that the finiteness of the cover can be removed in the special case that the common cover is a homogeneous Riemannian manifold; see \cite[Proposition 3.4]{GR03} and its correction in the errata \cite{GR21}.

In \cite{DGGW08}, Dryden, the second author, Sarah Greenwald and Webb extended the heat trace asymptotics to the case of arbitrary (not necessarily good) closed Riemannian orbifolds expressing the invariants $c_k(\O)$ as sums of contributions from the various \emph{primary} strata of the orbifold.   (See Notation and Remarks~\ref{notarem:isomax} for the definition of primary.)   As in the good case, the small-time heat trace expansion is of the form $(4\pi t)^{-d/2} \sum_{k=0}^{\infty} c_k(\mathcal{\O}) \, t^{k/2}$.    Primary strata of even, respectively odd, codimension in $\O$ contribute to coefficients $c_k(\O)$ for $k$ even, odd respectively.   As a consequence:

\begin{thm}[{\cite[Theorem 5.1]{DGGW08}}] \label{thm:dggw5.1}
A closed Riemannian orbifold that contains a primary singular stratum of odd codimension cannot be $0$-isospectral to any closed Riemannian manifold.
\end{thm}

Sean Richardson and the last author \cite{RS20} proved that an orbifold $\O$ is locally orientable if and only if $\O$ does not contain any primary singular strata of odd codimension.     They thus concluded that the 0-spectrum determines local orientability of an orbifold.
In addition,
\cite[Theorem 5.15]{DGGW08} shows that the $0$-spectrum distinguishes closed, locally orientable, 2-orbifolds with non-negative Euler characteristic from smooth, oriented, closed surfaces (in fact, a stronger result is proven there, see Remark~\ref{rem:orbisurfaces}).

\bigskip
This paper is organized as follows: Section \ref{sec:background} provides background on Riemannian orbifolds, their singular strata, differential forms and the Hodge Laplacian. In Section \ref{sec:heattrace}, we construct the heat kernel, heat trace, and heat invariants for $p$-forms on orbifolds.  Our construction follows the construction in the manifold case by Matthew Gaffney \cite{G58} and Vijay Kumar Patodi \cite{P71}, uses results of Donnelly and Patodi \cite{DP77} for good orbifolds, and parallels the adaptations in \cite{DGGW08}. Section \ref{applications} contains the proof of Theorem~\ref{thm:main1V4}. 
 
 In an appendix, we outline the computation of the heat invariants for good orbifolds in \cite{D76} and \cite{DP77}, indicating aspects of the construction that are applicable to more general settings.

\section{Riemannian Orbifold Background}
\label{sec:background}

\subsection{Definitions and Basic Properties}\label{background}
In this section we recall the definition and basic properties of a Riemannian orbifold. For comprehensive information about orbifolds see the paper \cite{Scott83} by Peter Scott, as well as the texts \cite{Thurston97} by William Thurston and \cite{ALR07} by Alejandro Adem, Johann Leida, and Yongbin Ruan. Here we follow a somewhat abridged form of the presentation given in \cite{G12} and \cite{DGGW08}.

\begin{definition}
\label{defn:ofld}~
Let $X$ be a second countable Hausdorff space.
\begin{enumerate}
\item For a connected open subset $U \subseteq X$, an \emph{orbifold chart} (of dimension $d$) over $U$ is a triple $\cc$ where $\wtu \subseteq \R^d$ is a connected open subset, $G_U$ is a finite group acting on $\wtu$ effectively and by diffeomorphisms, and $\pi_U\colon \wtu \to X$ is a map inducing a homeomorphism $\wtu/G \xrightarrow{\cong} U$. The open set $U$ is sometimes referred to as the image of the chart.

\item \label{it:injection} If $U \subseteq V \subseteq X$, an orbifold chart $\cc$ is said to \emph{inject} into an orbifold chart $\cc[V]$ if there exists a smooth embedding $i\colon \wtu \to \widetilde{V}$ and a monomorphism $\lambda\colon G_U \to G_V$ such that $\pi_U = \pi_V \circ i$ and $i \circ \gamma = \lambda(\gamma) \circ i$ for all $\gamma \in G_U$.

\item \label{it:ro} Orbifold charts $\cc$ and $\cc[V]$ on open sets $U$ and $V$ are said to be \emph{compatible} if for every $x\in U\cap V$, there exists a neighborhood $W\subset U\cap V$ of $x$ that admits an orbifold chart injecting into both $\cc$ and $\cc[V]$.   A $d$-dimensional \emph{orbifold atlas} on $X$ consists of a collection of compatible $d$-dimensional orbifold charts whose images cover $X$.

An \emph{orbifold} is a second countable Hausdorff topological space together with a maximal $d$-dimensional atlas.

\item A \emph{Riemannian structure} on an orbifold $\O$ is an assignment of a $G_U$-invariant Riemannian metric $\wtu$ to each orbifold chart $\cc$, compatible in the sense that each injection of charts as in part~\ref{it:injection} is an isometric embedding.

\item If an orbifold is the quotient space of a manifold under a
smooth proper action of a discrete group it is called a \emph{good} orbifold. Otherwise, it is called a \emph{bad} orbifold.
\end{enumerate}
\end{definition}

\begin{definition}\label{def:iso}
Let $\orb$ be an orbifold and let $x \in \orb$.
\begin{enumerate}

\item \label{it:welldef} We define the isotropy type of $x$ as follows:   A chart $\cc$ about $x$ defines a smooth action of $G_{U}$ on $\widetilde U \subset \mathbb R^d$. Fix a lift $\widetilde x \in \widetilde{U}$ of $x$ and let $\Iso(\widetilde{x})$ be the isotropy subgroup of $G_U$ at $\widetilde{x}$.   The map $\gamma\mapsto  d\gamma_{\widetilde{x}}\in \End(T_{\widetilde{x}}\widetilde{U})$, defines an injective linear representation of $\Iso(\widetilde{x})$.    Every finite-dimensional linear representation of a compact Lie group is equivalent to an orthogonal representation, unique up to orthogonal equivalence.    Thus $\Iso(\widetilde{x})$ can be viewed as a subgroup of the orthogonal group $O(d)$, unique up to conjugacy.    The conjugacy class of $\Iso(\widetilde{x})$ in $O(d)$ is independent both of the choice of the lift $\widetilde{x}$ of $x$ in $\widetilde{U}$ and of the choice of chart $\cc$ and is called the \emph{isotropy type} of $x$.

\item We say that $x$ is a \emph{regular point} if it has trivial isotropy type, and a \emph{singular point} otherwise.

\end{enumerate}
\end{definition}

We conclude our discussion of the basic properties of orbifolds by describing their singular stratification.

\begin{notarem}\label{notarem:isomax} Let $\orb$ be an orbifold.
\begin{enumerate} 
\item Define an equivalence relation on $\orb$ by saying that two points in $\mathcal O$ are \emph{isotropy equivalent} if they have the same isotropy type. The connected components of the isotropy equivalence classes of $\orb$ form a smooth stratification of $\orb$.    The strata whose corresponding isotropy types are non-trivial are called \emph{singular strata}.  (In the literature, the requirement that a singular stratum be connected is sometimes dropped.)

\item Every orbifold chart $\cc$ on $\orb$ also admits a smooth stratification whose strata are the connected components of the isotropy equivalence classes, where $\widetilde{x}$ and $\widetilde{y}$ in $\wtu$ are defined to be isotropy equivalent if $\Iso(\widetilde{x})$ and $\Iso(\widetilde{y})$ are conjugate in $G_U$.  We will refer to the strata of $\widetilde{U}$ that have non-trivial isotropy as \emph{pre-singular strata} and their union as the \emph{pre-singular set} of $\widetilde{U}$.  If $N$ is any stratum of $\orb$ that intersects $U$, then $\pi_U^{-1}(N\cap U)$ is a disjoint union of finitely many pre-singular strata of $\widetilde{U}$ each of which maps diffeomorphically to $N\cap U$ by $\pi_U$.

\item  Let $N$ be a singular stratum in $\orb$ and let $\iiso(N)<O(d)$ denote a representative of its isotropy type, unique up to conjugacy.  We will refer to $\iiso(N)$ as the \emph{isotropy group} of the stratum.   We denote by $\isomax(N)$ the set of all elements $\gamma \in \iiso(N)$ such that the dimension of the $1$-eigenspace of $\gamma$ is equal to the dimension of $N$.  We say that a stratum $N$ is \emph{primary} if $\isomax(N)$ is non-empty. (Item 4 below shows that this definition is equivalent to the one given in \cite{DGGW17}.)

 \item Similarly, if $W$ is a pre-singular stratum of an orbifold chart $\cc$, we let $\iiso(W)$ denote the common (up to conjugacy) isotropy group of its elements, we denote by $\isomax(W)$ the subset of those $\gamma\in \iiso(W)$ for which $W$ is open in the fixed point set of $\gamma$, and  we say that $W$ is \emph{primary} if $\isomax(W)$ is non-empty.     In particular, a pre-singular stratum $W$ of $\widetilde{U}$ is primary if and only if $\pi_U(W) $ is a (necessarily open) subset of a primary singular stratum of $\O$.

\end{enumerate}
\end{notarem}

An orbifold $\orb$ may contain some singular strata that are not primary. (See Example 2.16 in \cite{DGGW17}.)   However, the union of the primary strata is dense in the singular set as the following proposition shows.

\begin{prop}  Let $\orb$ be an orbifold.    Then:
\begin{enumerate}  
\item Every singular stratum of minimal codimension in $\orb$ is necessarily primary.   
\item Any component $C$ of the singular set that contains a non-primary singular stratum $N$ must also contain at least three primary singular strata. 
\end{enumerate}
\end{prop}

\begin{proof}
The first statement is immediate.  To prove the second, 
it suffices to prove the analogous statement for the connected components of the pre-singular set of any orbifold chart $\cc$.   Let $C$ be such a connected component that contains a non-primary pre-singular stratum $W$.   Each $\gamma \in \iiso(W)$ must lie in $\isomax(W')$ for some primary stratum $W' \subset C$ of lower codimension.  $\iiso(W)\nsubseteq \iiso(W')$ since  $\codim(W')< \codim(W)$, so the elements of $\iiso(W)$ must be distributed among at least two primary strata $W'$ and $W''$ of $C$.   In particular, we can choose $\gamma' \in \iiso(W)\cap \iiso(W')$ and $\gamma''\in \iiso(W)\cap\iiso(W'')$ such that $\gamma',\gamma''\notin \iiso(W')\cap \iiso(W'')$.    These conditions imply that $\gamma'\gamma''$ cannot lie in either $\iiso(W')$ or $\iiso(W'')$ and thus there must be a third primary stratum $W'''$ in $C$.
\end{proof}

\begin{ex}
By considering all the finite subgroups of the orthogonal group $O(2)$, we obtain a classification of the isotropy types of singular points that can occur in dimension $2$:
\begin{enumerate}
\item A rotation about the origin in $\R^2$ through an angle $\frac{2\pi}{k}$, for some integer $k\geq2$, generates a finite cyclic group of order $k$. The image of the origin in the quotient of $\R^2$ under this action is an isolated singular point called a \emph{cone point} of order $k$.
\item A reflection across a line through the origin generates a cyclic group of order $2$. The fixed points of the reflection correspond to a $1$-dimensional singular stratum in the quotient called a \emph{mirror} or \emph{reflector edge}.
\item A dihedral group generated by a pair of reflections across lines forming an angle $\frac{\pi}{k}$, for $2\leq k\in \mathbb Z^+$, yields a $0$-dimensional stratum in the quotient called a \emph{corner reflector} or \emph{dihedral point}. The corner reflector is not an isolated singular point, rather it is the point
where two mirror edges intersect.
\end{enumerate}
\end{ex}

\begin{ex}
The direct product $\O\times \O'$ of two orbifolds is again an orbifold.    For example, if $\O$ is a tear drop (an orbifold with a single cone point $x_0$ and underlying topological space the 2-sphere) and $M$ is a closed manifold, then $\O\times M$ is a bad orbifold with a unique singular stratum $\{x_0\} \times M$ of codimension $2$.
\end{ex}

\begin{nota}\label{nota:Rstuff} For $\gamma\in O(d)$, let $r$ be the multiplicity of the eigenvalue $-1$ if it occurs, and let $\{e^{\pm i \theta_j}\}_{j=1}^s$, where $\theta_1, \dots, \theta_s \in (0,\pi)$, be those eigenvalues of $\gamma$ with non-zero imaginary part, each repeated according to its multiplicity.   The expression $E(\theta_1, \theta_2, \dots, \theta_s;r)$ will be called the \emph{eigenvalue type} of $\gamma$.  When $r=0$ we write $E(\theta_1, \theta_2, \dots, \theta_s;)$ and when $s=0$ we write $E(;r)$. 

 The dimension of the $+1$ eigenspace of $\gamma$ is thus $d-2s-r$;  in particular, the parity of $r$ determines the parity of the codimension of the fixed point set of $\gamma$ in $\R^d$.
\end{nota}

\begin{remark} We take $\theta_j \neq \pi$ for all $1 \le j \le s$ for the notational convenience of letting $r$ count the total number of eigenvalues equal to $-1$. This convention can be dropped without affecting the results of Section~\ref{applications}.
\end{remark}

\subsection{Differential forms and the Hodge Laplacian}

We give a brief description of how to define differential forms on orbifolds.
For details of the constructions we refer to \cite{C19} or \cite{W12}. We take $\orb$ to be a $d$-dimensional orbifold throughout this section.

\begin{defn}\label{def:pform}
\hspace{1in}
\begin{enumerate}
\item As in the case of manifolds, a $p$-form $\omega$ on an orbifold $\O$ is defined to be a section of the $p$-th exterior power of the cotangent bundle $T^*(\orb)$.   The cotangent bundle and its exterior powers are examples of orbibundles; see, for example, \cite{C19} or \cite{G12} for an expository introduction to orbibundles or the references therein for more in-depth explanations.   We will only use the following:   If $U\subset \O$ is an open set on which there exists a coordinate chart $(\widetilde U, G_U, \pi_U)$, then any $p$-form $\omega$ on $U$ corresponds to a $G_U$-invariant $p$-form $\widetilde{\omega}$ on $\widetilde U$, which we call the \emph{lift} of $\omega$ to $\tu$. The differential form $\omega$ is said to be of \emph{class $C^k$} on $U$ if $\widetilde{\omega}$ is of class $C^k$.  The compatibility condition on overlapping charts in Definition~\ref{defn:ofld} part~\ref{it:ro} leads to a natural compatibility condition on the lifts of $p$-forms.   One can use partitions of unity to construct globally defined $p$-forms from $p$-forms defined on an open cover of $\O$.

\item For a Riemannian manifold $M$, recall that the  \emph{Hodge Laplacian $\Delta^p$} acting on the space of smooth $p$-forms $\Omega^p(M)$ is given by
\begin{align}
\label{defhodge}
\Delta^p\coloneqq -(d \delta + \delta d),
\end{align}
where $d$ is the exterior derivative and $\delta$ is the codifferential operator, i.e.\ the formal adjoint of $d$ obtained via the Riemannian metric.  It is straightforward to check that the Hodge Laplacian commutes with isometries.   The notion of Hodge Laplacian extends to Riemannian orbifolds $\O$ as follows:   Let $\omega$ be a smooth $p$-form on $\O$.  On each coordinate chart $\cc$, the $G_U$-invariance of the lift $\widetilde{\omega}$ and the fact that the Hodge Laplacian on Riemannian manifolds commutes with isometries together imply that $\Delta^p(\widetilde{\omega})$ is $G_U$-invariant.   We define $\Delta^p(\omega)$ on $U$ by
$$\widetilde{\Delta^p(\omega)}=\Delta^p(\widetilde{\omega}).$$
We can again use the fact that isometries between Riemannian manifolds intertwine the Hodge Laplacians along with the compatibility condition on orbifold charts to conclude that $\Delta^p(\omega)$ is well-defined on $\O$.
\end{enumerate}
 \end{defn}

\begin{remark}\label{hodgestar-adjoint} The results in the present work apply to both orientable and nonorientable orbifolds and manifolds.  In what follows the notation $d V_{M}$ indicates the Riemannian volume element in the orientable case, and the Riemannian density in the nonorientable case.
\end{remark}

It is known, see e.g.\ \cite[Theorem 4.8.1]{B99}, that the operator $\Delta^p$ viewed as an unbounded operator acting in $ L_p^2(\orb)$, the space of differential $p$-forms whose components are square-integrable functions, is essentially self-adjoint and has a purely discrete spectrum. By abuse of notation, in what follows we denote its closure by $\Delta^p$, which is then self-adjoint with compact resolvent. The spectral theorem thus applies to
$\Delta^p$ and we denote its eigenvalues by $0\leq\lambda_1 \leq \lambda_2\leq\dots \to +\infty$, with associated smooth eigenforms, denoted by $(\varphi_i)_i$, which form an orthonormal $ L_p^2$-basis.

We conclude this section by reviewing the notion of multi $p$-forms.
\begin{nota}\label{nota:doubleform}~
\begin{enumerate}
\item Recall that if $\pi: B\to M$ and $\pi':B'\to M'$ are vector bundles over manifolds $M$ and $M'$, then the external tensor product $B\boxtimes B'\to M\times M'$ is the vector bundle whose fiber over the point $(m,m')$ is $\pi^{-1}(m)\otimes \pi^{-1}(m')$.    Given a $C^\infty$ manifold $M$, let $M^k$ denote the $k$-fold Cartesian product $M\times \dots\times M$.  Sections of the $k$-fold external tensor product $\boxtimes^k(\wedge^p T^*(M))\to M^k$ are called multi $p$-forms on $M$ of order $k$.    In particular, a multi $p$-form of order $1$ is simply a $p$-form.  Those of order $2$ or $3$ are called double or triple $p$-forms, respectively.
Thus for example, if $(\mathcal U,x)$ and   $(\mathcal V,y)$ are local coordinate charts on $M$ and $F$ is a double $p$-form on $M$, then $F_{|\mathcal U\times \mathcal V}$ can be expressed as $\sum_{I, J} a_{I,J} dx^I \otimes dy^J,
$
where $a_{I,J} \in C^\infty(\mathcal U \times \mathcal V)$.  (Here $I$ and $J$ vary over $p$-tuples $ 1\leq i_1<\dots <i_p\leq d=\dim(M).$)

\item \label{it:gammasigma}  Given smooth manifolds $M$ and $N$ and smooth maps $\gamma: M\to M$ and $\sigma :N\to N$, denote by $\gamma_{1,\cdot}$ and $\sigma_{\cdot,2}$ the maps $M\times N\to M\times N$ given by $(a,b)\mapsto (\gamma(a), b)$ and $(a,b)\mapsto (a, \sigma(b))$.    If $M=N$, we also define $\gamma_{1,2}: M^2\to M^2$ by $(a,b)\mapsto (\gamma(a),\gamma(b))$.    We use analogous notation for maps of higher order Cartesian products.  Thus, for example, the map $\gamma_{1,\cdot,3}: M^3\to M^3$ is given by $(m_1,m_2,m_3)\mapsto (\gamma(m_1), m_2,\gamma(m_3)).$

\item \label{it:multipformofld} Analogous to the case of $p$-forms in Definition~\ref{def:pform}, we can use the notion of orbibundle to extend the definition of multiple $p$-forms directly to the orbifold setting, but we will not need the orbibundle formalism in what follows.   If $U$ and $V$ are open sets in an orbifold $\O$ on which there exist orbifold coordinate charts $\cc$ and $(\widetilde{V},G_V, \pi_V)$, then the restriction to $U\times V$ of a double $p$-form $F$ on $\O$ lifts to a section $\widetilde{F}$ of $\wedge^p T^*(\tu)\boxtimes \wedge^p T^*(\widetilde{V})$ that is invariant under pullback by each of the maps $\gamma_{1,\cdot}$ and $\sigma_{\cdot, 2}$ for $\gamma\in G_U$ and $\sigma\in G_V$.    Again, the compatibility condition on overlapping charts in Definition~\ref{defn:ofld} part~\ref{it:ro} of an orbifold leads to a natural compatibility condition on the lifts of double $p$-forms.  One can use partitions of unity to construct double $p$-forms on $\O$ from ones on an open cover of $\O\times\O$.   Multiple $p$-forms of any order are defined similarly.

\end{enumerate}

\end{nota}

Observe that if $F$ and $H$ are multi $p$-forms, say of orders $k$ and $\ell$, respectively, then $F\otimes H$ is a multi $p$-form of order $k+\ell$.

\begin{definition}\label{defcontraction}
Recall that if  $V_1,\dots,V_k$ are inner product spaces and if $1\leq i <j \leq k$ with $V_i=V_j$, then the contraction
$
C_{i,j}: V_1\otimes\cdots\otimes V_n\to V_1\otimes \cdots \otimes\widehat{V_i}\otimes \cdots\otimes\widehat{V_j}\otimes\cdots\otimes V_k $ is the linear map whose value on decomposable vectors is given by
$$\alpha_1\otimes \cdots\otimes \alpha_n \mapsto \langle \alpha_i, \alpha_j\rangle \alpha_1\otimes\cdots\otimes \widehat{\alpha_i}\otimes\cdots\otimes \widehat{\alpha_j}\otimes \cdots \otimes \alpha_k$$ where $\langle\,,\,\rangle$ is the inner product on $V_i$.  (Here we use hat notation such as $\widehat{V_i}$ or $\widehat{\alpha_i}$ to indicate that the corresponding factor is removed from the product.)

In particular, let $(M,g)$ be a Riemannian manifold, $F$ a multi $p$-form of order $k$ on $M$, and suppose $m=(m_1,\dots, m_k)\in M^k$ satisfies $m_i=m_j$.    Then we can use the inner product on $\wedge^pT^*_{m_i}(M)$ to define the contraction $C_{i,j}(F(m)).$

To define the contraction of such a multi $p$-form in the setting of an orbifold $\O$ at a point $m=(m_1,\dots, m_k)\in \O^k$ with $m_i=m_j$, let $\widetilde{F}$ be a local lift on a neighborhood $U_1\times \dots \times U_k$ as in Notation~\ref{nota:doubleform} part~\ref{it:multipformofld}.   Choose a lift $\widetilde{m}=(\widetilde{m}_1,\dots, \widetilde{m}_k)$ in such a way that $\widetilde{m}_i=\widetilde{m}_j$.  We can then define $C_{ij}(F(m))$  by the condition that it lifts to $C_{ij}(\widetilde{F}(\widetilde{m}))$.   This definition is independent of the choice of $\widetilde{m}$ (subject only to the condition $\widetilde{m}_i=\widetilde{m}_j$), and it is independent of the choice of chart since the compatibility condition on charts respects the Riemannian structure.
\end{definition}

\section{Heat trace for differential forms on orbifolds}
\label{sec:heattrace}

In Subsection~\ref{subsec:kernel} we will construct the fundamental solution of the heat equation for the Hodge Laplacian on arbitrary closed Riemannian orbifolds by first constructing a parametrix.  Our construction follows and generalizes the construction in the manifold case by Gaffney \cite{G58} and Patodi \cite{P71}.  (The assumption of orientability in the work of Gaffney and Patodi is not needed.)  In Subsection~\ref{subsec:Don}, we translate a result of Donnelly into our context and employ it
to develop the small-time asymptotic expansion of the heat trace.   Our arguments are similar to those used in \cite{DGGW08} to construct the heat kernel and heat trace asymptotics for the Laplacian acting on functions on a closed Riemannian orbifold.

\subsection{Heat kernel for $p$-forms on manifolds and orbifolds}\label{subsec:kernel}

\begin{definition}\label{fundsol}
We say that $K^p: \mathbb (0,\infty) \times  \O \times  \O \to \Lambda^p( T^* \O) \otimes \Lambda^p (T^* \O)$ is the \emph{heat kernel} or \emph{fundamental solution of the heat equation} for $p$-forms if it satisfies
\begin{enumerate}
\item $K^p(t,\cdot,\cdot)$ is a double $p$-form  for any $t\in (0,\infty)$;
\item $K^p$ is continuous in the three variables, $C^1$ in the first variable and $C^2$ in the second variable;
\item $(\partial _t + \Delta^p_x ) K^p(t,x,y)= 0$ for each fixed $y\in \O$, where $\Delta^p_x$ is the Hodge Laplacian with respect to the variable $x$;
\item\label{initcond} For every continuous $p$-form $\omega$ on $\O$ and for all $x\in \O$, we have
$$\lim_{t\to 0^+}\,\int_\O\,C_{2,3} K^p(t,x,y)\otimes \omega(y) \, dV_\O(y) = \omega(x).$$
\end{enumerate}
(Because the variable $t$ in the preceding definition is a real parameter, in contrast to a space variable, it is not counted when specifying indices in expressions using Notation~\ref{nota:doubleform} part~\ref{it:gammasigma}, or in expressions involving contractions using the notation from Definition~\ref{defcontraction}. Also, the variables over which a contraction is occuring will be repeated.  For example, in part~\ref{initcond} above, $C_{2,3} K^p(t,x,y)\otimes \omega(y)$ indicates contraction in the second and third entries, not counting $t$.)
\end{definition}

The same argument as in the manifold case shows that if the heat kernel exists, then it is unique and given by
\begin{equation}\label{fundsoldvp}
K^p(t,x,y)= \sum_{j= 1}^{\infty} \varphi_j\otimes\varphi_j(x,y) e^{-\lambda_j t}
\end{equation}
where $\{\varphi_j\}_{j=1}^\infty$ is an orthonormal basis of $L_p^2(\O)$ consisting of eigenfunctions and the $\lambda_j$ are the corresponding eigenvalues as above.  In particular, the heat kernel is invariant under any isometry $\gamma$, i.e.,
\begin{equation}(\gamma^*_{1,2} K^p)(t,x,y)=K^p(t,x,y)\end{equation}
in the notation of Notation~\ref{nota:doubleform} part~\ref{it:gammasigma}.

\begin{defn}\label{def.param}
A \emph{parametrix} for the heat operator on $p$-forms on $\O$ is a function
$H:(0,\infty)\times\O\times\O\to \Lambda^p(T^* \O) \otimes \Lambda^p(T^*\O)$  satisfying:
\begin{enumerate}
\item $H(t,\cdot,\cdot)$ is a double $p$-form for each $t\in \mathbb (0,\infty)$;
\item $H$ is  $C^\infty$ on $(0,\infty)\times\O\times\O$;
\item $(\partial_t + \Delta^p_x) H(t,x,y)$ extends continuously to $[0,\infty)\times\O\times\O$;
\item For every continuous $p$-form $\omega$ on $\O$ and for all $x\in \O$, we have $$\lim_{t\to 0^+}\,\int_\O\,C_{2,3} H(t,x,y)\otimes\omega(y)\,dV_\O(y) = \omega(x).$$
\end{enumerate}
\end{defn}

We first recall the construction of a parametrix by Patodi, see also Gaffney \cite{G58}.

\begin{prop}[{\cite{P71}}]
\label{localparam}
There exist smooth double $p$-forms $u^p_i$ for $i=0,1,2,\dots$ defined on a neighborhood of the diagonal of $M\times M$ in any Riemannian manifold $M$ of dimension $d$ (more precisely, $u^p_i(x,y)$ is well-defined whenever $y$ is in a normal neighborhood about $x$) satisfying the following.

\begin{enumerate}
\item \label{uzero}
If $\{\omega_1,\dots, \omega_r\}$, where  $r=\binom{d}{p}$, is an orthonormal basis of $\Lambda^p\,(T^*_xM)$, then
$$u_0^p(x,x)=\sum_{j=1}^r\, \omega_j\otimes \omega_j.$$
(Note that this expression is independent of the choice of orthonormal basis).   Equivalently under the identification of $\Lambda^p\,(T^*_xM)\otimes \Lambda^p\,(T^*_xM)$ with $\End(\Lambda^p\,(T^*_xM)) $,
$u^p_0(x,x)$ is the identity endomorphism of $\Lambda^p( T_x^*M)$.
\item \label{um} For each $m=1,2,\dots$, we have
$$(\partial_t + \Delta^p_x) H^{(m)}(t,x,y)= (4\pi)^{-d/2}e^{-d^2(x,y)/4t}t^{m-\frac{d}{2}}\Delta^p_x u^p_m(x,y),$$
 where
$$H^{(m)}(t,x,y)\coloneqq \shapl\ehapl,$$
and $d^2(x,y)$ is the square of the Riemannian distance between $x$ and $y$.
In particular, if $\psi:M\times M\to \R$ is supported on a sufficiently small neighborhood of the diagonal and is identically one on a smaller neighborhood of the diagonal, then  $\psi H^{(m)}$ is a parametrix for the heat operator on $p$-forms when $m>\frac{d}{2}$.

\item \label{traceinv} The $u_i^p$ are the unique double $p$-forms satisfying conditions~\ref{uzero} and \ref{um}.  If $\gamma$ is an isometry of an open set in the domain of $u^p_i $, then $\gamma^*_{1,2}u^p_i=u^p_i$.
\end{enumerate}
\end{prop}

\begin{notarem}\label{notarem:hm} Fix $\e>0$ so that each $x\in\O$ is the center of a ``convex geodesic ball'' $W$ of radius $\epsilon$.   By this we mean that W admits a chart $(\widetilde{W},G_W,\pi_W)$, where $\widetilde{W}$ is a convex geodesic ball of radius $\epsilon$ centered at the necessarily unique point $\tx$ satisfying $\pi_W(\tx)=x$.   (In particular, $\tx$ has isotropy $G_W$.)     Cover $\O$ by finitely many such geodesic balls $\wa$ with charts
$(\twa, G_\a,\pi_\a)$, $\alpha=1,\dots, s$.  (Here we write $G_\a$ for
$G_{\wa}$ and $\pi_\a$ for $\pi_\wa$.)   Let $\tua\subset \twa$,
respectively $\tva\subset\twa$, be the concentric geodesic balls of radius $\frac{\e}{4}$, respectively
$\frac{\e}{2}$, and let $\ua=\pi_\a(\tua)$ and $\va=\pi_\a(\tva)$.  We may assume that the family of balls
$\{\ua\}_{1\le\a\le s}$ still covers $\O$.

For each $\a$ and each nonnegative integer $m$, we define  $$\tham:
(0,\infty)\times\twa\times\twa \rightarrow  \Lambda^p(\twa) \otimes \Lambda^p(\twa)$$ by
\begin{align}
\label{nota.hm}
\tham(t,\tx,\ty)\coloneqq\sha\ehap,
\end{align}
where the $u^p_i$ are the double $p$-forms defined in Proposition~\ref{localparam}.
Since each $\g\in\ga$ is an isometry of $\twa$, we have
$\gamma^*_{1,2}u^p_i=u^p_i$.

It follows that the double $p$-form (depending on the parameter $t$)

$$\sum_{\gamma\in\ga}\gamma^*_{\cdot,2}\tham$$
is $\ga$-invariant in both $\tx$ and $\ty$ and
thus descends to a well-defined function on
$(0,\infty)\times\wa\times\wa$, which we denote by $\ham$.

Let $\psia:\O\to\R$ be a $C^\infty$ cut-off function  which is identically one
on $\va$
and is supported in $\wa$.  Let $\{\etaa :\a=1,\dots,s\}$ be a partition of
unity on
$\O$ with the support of $\etaa$ contained in  $ \overline{\ua}$.   Define $\hm$ on $(0,\infty)\times\O\times\O$
by
\begin{equation}\label{eq.hm}
\hm(t,x,y)\coloneqq\sum_{\a
=1}^s\,\psia(x)\etaa(y)\ham(t,x,y).
\end{equation}

\end{notarem}
\begin{prop}\label{prop.param}
$\hm$ is a parametrix for the heat kernel on $\O$ when
$m>\frac{d}{2}$.

Moreover, the extension of $\heatp \hm(t,x,y)$ to $[0,\infty)\times\O\times\O$ is of class $C^k$ if $m>\frac{d}{2}+k$ for any $k\in \mathbb N$.

\end{prop}
\begin{proof}
The proof is similar to that carried out in \cite{DGGW08} (see also references therein) for the case of the Laplacian acting on functions.   The first three properties of a parametrix and the final statement of the proposition are straightforward.     We include here the proof of the last property of a parametrix (the reproducing property).

Let $\omega$ be a continuous $p$-form on $\O$.
Let $\tpsia$ and $\tetaa$ be the lifts of $\psia$ and $\etaa$ to
$\twa$.
We consider the lifts so as to make use of Proposition~\ref{localparam}
part~\ref{um} and we also use the fact that the properties of a parametrix hold locally (see \cite[Remark 3.7]{DGGW08}).

Since $\supp(\etaa)\subset\overline{\ua}\subset\wa$, we have
\begin{equation}
    \label{eq.ham}\begin{aligned}
&\pi_\a^*\int_\O\psia(x)\etaa(y)C_{2,3}\ham(t,x,y)\otimes \omega(y)\,dy\\&
=\frac{\psia(x)}{|G_\a|}\sum_{\g\in\ga}\int_\twa\,\tetaa(\ty)C_{2,3}\g_{.,2}^*\tham(t,\tx,\ty)\otimes
\widetilde{\omega}_\alpha(\ty) \, d\ty,
\end{aligned}
\end{equation}
where $\widetilde{\omega}_\alpha$
  is the pull-back of $\omega_{|\twa}$ to
$\twa$ and $\tx$ is an arbitrarily chosen point in the preimage of $x$
under the map $\twa\to\wa$. We change variables in each of the integrals in the right-hand side of
Equation~\eqref{eq.ham}, letting $\tul=\g^{-1}(\ty)$.  Since
$\g$ is an isometry and since $\tetaa$ and $\widetilde{\omega}_\alpha$
are $\g$-invariant, each integral in the summand is equal to
\begin{equation}\label{eq.tham}\int_\twa\,\tetaa(\tul)C_{2,3} \tham(t,\tx,\tul)\otimes
\widetilde{\omega}_\alpha(\tul) \,d\tul.\end{equation}
As $\too$, the
integral \eqref{eq.tham} above converges to
$\tetaa(\tx)\widetilde{\omega}_\alpha(\tx)=\etaa(x)\pi_\a^*(\omega(x))$.  Noting
that $\psia\equiv 1$ on the support of $\etaa$, it follows that  both sides of
Equation~\eqref{eq.ham} converge to $\etaa(x)\pi_\a^*(\omega(x))$ as $\too$.    Since both sides of
Equation~\eqref{eq.ham} are identically zero when $x$ lies outside of
$\supp(\psia)\subset \wa$, the left-hand side converges to
$\etaa(x)\pi_\a^*(\omega(x))$ for each fixed $x\in \O$.
Thus
\[
\lim_\too\int_\O\psia(x)\etaa(y)C_{2,3} \ham(t,x,y)\otimes \omega(y)\,dy =\etaa(x)\omega(x),
\]
  and by  Equation~\eqref{eq.hm} we have
\[
\lim_\too\int_\O\,\hm(t,x,y)\omega(y)\,dy=\sum_\a\,\etaa(x)\omega(x)=\omega(x). \qedhere
\]
\end{proof}

\begin{nota}\label{conv} ~
\begin{enumerate}
\item For $A$ and $B$ continuous double $p$-form valued functions on $[0,\infty)\times\O\times\O$, we define the convolution $A*B$ on $(0,\infty)\times\O\times\O$ as
$$A*B(t,x,y)\coloneqq\int_0^t\,d\theta\int_\O\,C_{2,4} A(\theta,x,z)\otimes B(t-\theta,y,z) dV_{\mathcal O}(z).$$ 
\item Fix
$m>\frac{d}{2} +2$.  Define $\kappa_0(t,x,y)
:=\heat\hm(t,x,y)$ and, for $j=1,2,\dots$, set $\kappa_j(t,x,y)\coloneqq\kappa_0*\kappa_{j-1}$.
\end{enumerate}
\end{nota}

An argument analogous to that in the manifold case \cite{G58} shows that the series
\begin{equation}\label{lem.qm} P_m(t,x,y):=\sum_{j=1}^\infty\,(-1)^{j+1}\kappa_j(t,x,y)\end{equation}
converges uniformly and absolutely on $[0,T]\times\O\times\O$ for each $T>0$.  Thus $P_m$ is continuous.  Moreover, $P_m$ is of class $C^2$ on $\rp\times\O\times\O$ and for any $T>0$, there exists a constant $C$ such that
\begin{equation} P_m(t,x,y)\leq Ct^2\end{equation}
on $[0,T]\times\O\times\O$.
(Our indexing is different from that in \cite{G58}.   Our $\kappa_j$ is Gaffney's $\kappa_{j+1}$.)

We are now able to state the main theorem of this section, which follows in the same way as that in \cite{G58}.
\begin{thm}
\label{lem.conv}  Let $m>\frac{d}{2}+2$. Define  $P_m$ as in Equation~\eqref{lem.qm} and let
$$K^p(t,x,y) \coloneqq\hm(t,x,y)
+(P_m*\hm)(t,x,y).$$
Then $K^p$ is a heat kernel.
Moreover,
$$K^p(t,x,y)=\hm(t,x,y) +O(t^{m-\frac{d}{2}+1}), \text{ as } t\to 0^+.$$
\end{thm}

\begin{nota}\label{nota.ha}  Let
$$\tha(t,\tx,\ty)=\sum_{\g\in\ga}\shag\ehagi.$$
Observe that $\tha$ is $\ga$-invariant in both $\tx$ and $\ty$ and thus is the pull-back
of a double $p$-form valued function, which we denote by $\ha$,
on $(0,\infty)\times\wa\times\wa$.
\end{nota}

As a consequence of Theorem~\ref{lem.conv}, we have:

\begin{thm}
\label{thm.heattrace}  In the notation of Equation \eqref{eq.hm}
and Notation~\ref{nota.ha}, the trace of the heat kernel has an asymptotic expansion as
$\too$ given by
\[
\tr K^p(t):=\int_\O\,C_{1,2}K^p(t,x,x)\,dV_{\mathcal O}(x)
\sim\,\sum_{\a=1}^s\,\int_\O\,\etaa(x)C_{1,2}\ha(t,x,x)\,dV_{\mathcal O}(x).
\]
\end{thm}

\subsection{Heat invariants for $p$-forms on orbifolds}\label{subsec:Don}

We will use Theorem~\ref{thm.heattrace} together with a theorem of Donnelly \cite[Theorem 4.1]{D76} (see also \cite[Theorem 3.1]{DP77}) to compute the small-time asymptotics of the heat trace for the Laplacian on $p$-forms on closed Riemannian orbifolds.
Our presentation is similar to that in \cite{DGGW08} where the case $p=0$ is carried out.

\begin{notarem}\label{nota:don} ~
\begin{enumerate}
\item \label{it:dontriples} Let $\mathcal{T}$ be the set of all triples $(M,\gamma, a)$ where $M$ is a Riemannian manifold, $\gamma$ is an isometry of $M$ and $a$ is in the fixed point set $\Fix(\gamma)$ of $\gamma$.   A function $h:\mathcal{T}\to \R$ is said to satisfy the \emph{locality property} if $h(M,\gamma, a)$ depends only on the germs at $a$ of the Riemannian metric and of $\gamma$.     The function $h$ is said to be \emph{universal} if whenever
$\sigma:M_1\to M_2$ is an isometry between two Riemannian manifolds and $\gamma_i$ is an isometry of $M_i$ with $\gamma_2=\sigma\circ\gamma_1\circ\sigma^{-1}$, then $h(M_1,\gamma_1,a)=h(M_2,\gamma_2,\sigma(a))$ for all $a\in \Fix(\gamma_1)$.

\item We will also consider functions on the set $\mathcal{T}'$ of all $(M, \eta, \gamma, a)$ where $(M,\gamma, a)$ is given as in part~\ref{it:dontriples} and $\eta$ is a $\gamma$-invariant function on $M$.   The locality and universality properties are analogously defined.   (In the definition of locality, $h(M,\eta, \gamma, a)$ may also depend on the germ of $\eta$ at $a$.  In the definition of of universality, we have $h(M_1,\eta_1,\gamma_1,a)=h(M_2,\eta_2,\gamma_2,\sigma(a))$ if $\eta_2\circ\sigma=\eta_1$ and the other conditions in part~\ref{it:dontriples} hold.)

\item Let $M$ be a Riemannian manifold of dimension $d$ and let  $\gamma: M\to  M$ be an isometry.   Then each component of the fixed point set $\Fix(\gamma)$ is a totally geodesic, closed submanifold of $M$. (See \cite{K58}.) If $M$ is compact, then $\Fix(\gamma)$ has only finitely many components. In any Riemannian manifold, at most one component of $\Fix(\gamma)$ can intersect any given geodesically convex ball.  For $a\in \Fix(\gamma)$, note that the differential $d\gamma_a: T_a M \to T_a M$ fixes $T_aQ$ and restricts to an isomorphism of $T_a(Q)^\perp$, where $Q$ is the connected component of $\Fix(\gamma)$ containing $a$.
In particular, we have $\det(\Id_{\codim(Q)}-A_a)\neq 0$, where $A_a$ is the restriction of the action of $d\gamma_a$ to $(T_aQ)^\perp$ and $\codim(Q)$ denotes the codimension of $Q$.

\item \label{it:don-trace} For $\alpha$ in the orthogonal group $O(d)$, we denote by $\tr_p(\alpha)$ the trace of the natural action of $\alpha$ on $\wedge^p(\R^d)$.  For $M$ and $\gamma$ as in part~\ref{it:dontriples}, the Riemannian inner product on $T_a M$ allows us to identify $d\gamma_a$ with an element of  $O(d)$ and $\tr_p(d\gamma_a)$ coincides with the trace of $\gamma_a^*: \wedge^p(T_a^*M)\to  \wedge^p(T_a^*M)$.

\item \label{it:don-CFix} Let $M$ be a compact Riemannian manifold and  $\gamma: M \to  M$ be an isometry. We denote the set of connected components of the fixed point set of $\gamma$ by $\mathcal{C}(\Fix(\gamma))$.
\end{enumerate}

\end{notarem}

\begin{thm}\label{bkthm}\label{bklocal}
\hspace{1in}$\text{  }$
\newline\begin{enumerate}

\item \label{it:bkthm-global} \cite[Theorem 4.1]{D76}; \cite[Theorem 3.1]{DP77}. In the notation of Notation~\ref{nota:don}, there exist functions $b_k$, $k=0, 1, 2, \dots$, on $\mathcal{T}$ satisfying both the locality and universality properties such that if $M$ is a compact Riemannian manifold and  $\gamma: M \to  M$ is an isometry, then, as $t\rightarrow 0^+$,

\[
\int_{M} C_{1,2} \gamma^*_{\cdot, 2} K^p(t,x,x) \, dV_{M}(x) \sim \sum_{Q\in \mathcal{C}(\Fix(\gamma))} (4\pi t)^{-\dim(Q)/2} \sum_{k=0}^{\infty} t^k \int_{Q} b^p_k(\gamma,a) \, dV_Q (a).
\]
(Here we are writing $b_k(\gamma,a)$ for $b_k(M,\gamma, a)$.) Moreover,
\[
b^p_0(\gamma,a)=\frac{\tr_p(d\gamma_a)}{\lvert \det(\Id_{\codim(Q)}-A_a)\rvert}.
\]

\item \label{it:bkthmlocal} (Local version.)  There exist functions $c_k$, $k=0, 1, 2, \dots$, on $\mathcal{T}'$ satisfying both the locality and universality properties such that the following condition holds:  if $M$ is an arbitrary  Riemannian manifold, $\gamma$ is an isometry of $M$, and $\eta$ is a $\gamma$-invariant function on $M$ whose support is compact and geodesically convex, then as $t\to 0^+$, we have
\begin{multline*}
    \int_{M}\,(4\pi t) ^{-d/2} e^{-\frac{ d^2(x, \gamma(x))}{4t}}\eta(x)\left ( \sum_{j=0}^\infty t^j C_{1,2} \gamma^*_{\cdot, 2} u^p_j(x,x)\right) \, dV_{M}(x) \\
\sim \sum_{Q\in \mathcal{C}(\Fix(\gamma))} (4\pi t)^{-\dim(Q)/2} \sum_{k=0}^{\infty} t^k \int_{Q} c^p_k(\eta,\gamma,a) \, dV_Q (a),
\end{multline*}
where the $u^p_j$, $j=1,2,\dots$, are the double $p$-forms defined in Proposition~\ref{localparam}.  (We are writing $c_k(\eta,\gamma,a)$ for $c_k(M,\eta, \gamma, a).$   Note that there is at most one component $Q$ of $\Fix(\gamma)$ that intersects the support of $\eta$ and thus at most one non-zero term in the sum.)

Moreover, the dependence on $\eta$ is linear.  In particular, if $\eta\equiv 1$ near $a$, then $c_k^p(\eta, \gamma,a)=b_k^p(\gamma,a)$.
\end{enumerate}
\end{thm}

\begin{remark} In the first statement in the previous theorem, we have omitted the hypotheses in \cite{DP77} that $M$ is orientable and that $\gamma$ is orientation-preserving as they are not needed for the proof.

The proof of the second statement is a minor adaptation of Donnelly's proof of the first statement.   In the appendix, we will give an exposition of the proof, making clear that aspects of the statement apply to much more general integrals.
\end{remark}

Recall that for closed Riemannian manifolds, the heat trace has a small-time asymptotic expansion \begin{equation}\label{manifoldasymp}
\tr K^p(t) :=\int_M C_{1,2}K^p(t,m ,m)\,dV_M(m)\sim_{t\to 0^+}
(4\pi t) ^{-d/2} \sum_{i=0}^{\infty} a^p_i(M) t^i,
\end{equation}

where
\[
a^p_i( M)\coloneqq \int_{ M} C_{1,2} u^p_i(m,m)\, dV(m).
\]
The $a_i^p$ are called the \emph{heat invariants}.

The first two heat invariants for $p$-forms on manifolds are given by
\begin{align}
\label{a10}
    a_0^p({M}) &= \binom{d}{p} \vol ({M}), \\
    \label{a11}
    a_1^p({M}) &= \left(\frac{1}{6}\binom{d}{p}-\binom{d-2}{p-1}\right) \int_{{M}} \tau (m) \, dV_M(m),
    \end{align}
where $\tau$ is the scalar curvature, and we use the convention that $\binom{m}{n}=0$ whenever $n < 0$ or $n > m$. (See, for example, \cite{P70} and references therein.)

We now address the heat invariants for the Hodge Laplacian on closed Riemannian orbifolds.
\begin{nota}\label{invariants} Let $\O$ be a closed Riemannian orbifold of dimension $d$.
\begin{enumerate}
\item \label{it:invariants-Ip0t}  For $k=0,1,2,\dots$, define $U_k^p\in C^\infty(\O)$ as follows:    For $x\in \O$, choose an orbifold chart $(\widetilde{U},G_U, \pi_U)$ about $x$, let $\tx$ be any lift of $x$ in $\widetilde{U}$ and set $U_k^p(x)\coloneqq C_{1,2}u_k^p(\tx,\tx)$ where $u_k^p$ is the double $p$-form 
defined in Proposition~\ref{localparam}.  This definition is independent of both the choice of chart and the choice of lift $\tx$ since the $u_k^p$ are local isometry invariants.

Define
\begin{equation}\label{eq:heatinvM}
a_k^p(\O):=\int_\O\, U_k^p(x) \, dV_\mathcal O(x).
\end{equation}
One easily verifies that the resulting expressions for the $a_k^p(\O)$ are identical to those for manifolds.  E.g.,  $a_0^p({\O}) = \binom{d}{p} \vol ({O})$ and $a_1^p(\O)$ is given by Equation~\eqref{a11} with $M$ replaced by $\O$.

Set
$$I_0^p(t)\coloneqq(4\pi t)^{-d/2}\sum_{k=0}^\infty\,a_k^p(\O) t^k.$$

\item \label{it:invariants-IpNt}  Let $N$ be a stratum of the singular set of $\O$.    For $\gamma\in \isomax(N)$, define $b_k^p(\gamma,\cdot): N\to \R$ as follows:  For $a\in N$, let $(\widetilde{U}, G_U, \pi_U)$ be an orbifold chart about $a$ in $\O$ and let $\tilde{a}$ be a lift of $a$ in $\widetilde{U}$.   Note that $\tilde{a}$ lies in a pre-singular stratum $W$ of $\widetilde{U}$ (see Notation and Remarks~\ref{notarem:isomax}) and $\gamma$ is naturally identified with an element of $\isomax(W)$.     We define $$b_k^p(\gamma,a):=b_k^p(\gamma,\tilde{a}),$$ where the right-hand side is defined as in Theorem~\ref{bkthm} part~\ref{it:bkthm-global}.   The universality of the functions $b_k^p :\mathcal{T}\to \R$ (as stated in Theorem~\ref{bkthm} part~\ref{it:bkthm-global}) implies that  the left-hand side is independent both of the choice of chart and of the choice of lift $\tilde{a}$.   Define $b_k^p(N,\cdot): N\to \R$ by

$$ b_ k^p(N,a):=\sum_{\gamma\in \isomax(N)}\, b_k^p(\gamma,a),$$

and set
\begin{align}
\label{nota_heatb}
b_ k^p(N):=\sum_{\gamma\in \isomax(N)}\,\int_N\, b_k^p(\gamma,a)\,dV_N(a),
\end{align}
where $dV_N$ is the volume element on $N$ for the Riemannian metric on $N$ induced by that of $\O$.

Aside: this notation differs slightly from that in \cite{DGGW08} where the case $p=0$ is studied.

Set
$$I_N^p(t):=(4\pi t)^{-\dim(N)/2}\sum_{k=0}^\infty\,b_k^p(N) t^k.$$
Observe that $I_N^p(t)=0$ if $N$ is not a primary stratum.
\end{enumerate}
\end{nota}

\begin{remark}
Since each stratum $N$ consists of points of a fixed isotropy type, the expression for $b_0^p(\gamma, a)$ in Theorem~\ref{bkthm} part~\ref{it:bkthm-global} is independent of $a\in N$ and we denote it by $b_0^p(\gamma)$. In the notation of Equation~\eqref{nota_heatb}, we have
\begin{equation}\label{niceb0} b_0^p(N)=\vol(N)\sum_{\gamma\in\isomax(N)}\,b_0^p(\gamma).
\end{equation}
\end{remark}

\begin{thm}\label{thm.asympt} Let $\O$ be a closed $d$-dimensional Riemannian orbifold, let $p\in \{1,\dots, d\}$ and let $0\leq\lambda_1\leq\lambda_2\leq\dots \to +\infty$ be the spectrum of the Hodge Laplacian acting on smooth $p$-forms on $\O$.   The heat trace has an asymptotic expansion as $t\to 0^+$ given by
\begin{equation}
\label{eq:traceO}
 \tr K^p(t) \sim_{t\to 0^+}\, I_0^p(t)+\sum_{N\in S(\O)} \, \frac{I_N^p(t)}{|\iso(N)|},
 \end{equation}
where $S(\O)$ is the set of all singular $\O$-strata and where $|\iso(N)|$ is the order of the isotropy group of $N$.    This asymptotic expansion is of the form
\begin{align}
    \label{heatascoeffc}
(4\pi t)^{-d/2}\sum_{j=0}^\infty\, c^p_j(\mathcal O)\,t^{\frac{j}{2}}
\end{align}
with $c^p_j(\mathcal O)\in\R$.
\end{thm}
Observe that if there are no singular strata, then Equation~\eqref{eq:traceO} agrees with Equation~\eqref{manifoldasymp}.

\begin{proof} By Theorem~\ref{thm.heattrace}, we can write
\begin{equation}\label{prelim}
\tr K^p(t)
\sim_{\too}\,\sum_{\a=1}^s\,\int_\O\,\etaa(x)C_{1,2}\ha(t,x,x)\,dx.
\end{equation}
Recalling the notation introduced in Equation~\eqref{nota.hm}, Equation~\eqref{prelim}, we have
\begin{equation}\label{start}
\tr K^p(t) \sim_{\too}\sum_{\a=1}^s\frac{1}{|G_\a|}\sum_{\g\in\ga}L(t,\alpha,\gamma)
\end{equation}
where
\begin{equation}\label{start2}
L(t,\alpha,\gamma)\coloneqq \int_{\tua}(4\pi t)^{-d/2}e^{-d^2(\tx,\g(\tx))/4t}\tetaa(\tx) \notag\,
 C_{1,2}\Big(\g_{.,2}^*u^p_0(\tx,\tx)+ t\g_{.,2}^*u^p_1(\tx,\tx)+\cdots\Big) dV_{\tua}(\tx).
\end{equation}

We will group together various terms in this double sum.   First consider the identity element $1_\a$ of each $G_\a$.   By Notation~\ref{invariants} part~\ref{it:invariants-Ip0t}, we have
$$\sum_{\a=1}^s\frac{1}{|G_\a|}L(t,\a, 1_\a)= (4\pi t)^{-d/2}\int_\O (U_0^p(x) + tU_1^p(x) +\cdots) \,dV_\O(x)   =I_0^p(t).$$

Next let
\begin{equation}\label{I'} I'(t)\coloneqq\sum_{\a=1}^s\frac{1}{|G_\a|}\sum_{1_\a\neq\g\in\ga}L(t,\alpha,\gamma).
  \end{equation}
  It remains to show that
  \begin{equation}\label{toshow}I'(t)=\sum _{N\in S(\O)}\frac{I_N^p(t)}{|\iso(N)|}.\end{equation}
We will apply Theorem~\ref{bklocal} to each term in $I'(t)$.   For each $\a$ and each nontrivial element $\gamma\in G_\a$, Theorem~\ref{bklocal} (with $\tua$ playing the role of $M$ in Theorem~\ref{bklocal} part~\ref{it:bkthmlocal}) expresses $L(t,\a, \gamma)$  as a sum of integrals over the various components of the fixed point set of $\gamma$ in $\tua$.   Each such component is a union of pre-singular strata in $\tua$.  Since the union of those pre-singular strata $W$ in $\tua$ for which $\gamma\in \isomax(W)$ is an open subset of $\Fix(\gamma)$ of full measure, we can instead add up the integrals over such pre-singular strata.  Letting $S(\tua)$ denote the pre-singular strata of $\tua$ and applying Theorem~\ref{bklocal}, we have
\begin{equation}
\label{ip}
I'(t)=\sum_{\a=1}^s \sum_{W\in S(\tua)}\,\sum_{\g\in\isomax(W)}\,\frac{1}{|G_\a|}(4\pi t)^{-\dim(W)/2}\sum_{k=0}^{\infty}\,t^k\int_W\,c_k^p(\tetaa, \gamma, \tilde{a})\, dV_W(\tilde{a}).
\end{equation}

 Let $N\in S(\O)$.
For $a\in N$, set
$$c_k^p(\etaa,\gamma, a):= c_k^p(\tetaa, \gamma, \tilde{a})$$
 where $\tilde{a}$ is any lift of $a$ in $\tua$.   This is well-defined since any two lifts differ by an element of $G_\a$ and $\tetaa$ is $G_\a$-invariant.  Moreover, Theorem~\ref{bklocal} part~\ref{it:bkthmlocal} implies that
\begin{equation}\label{sumck}\sum_{\a=1}^s c_k^p(\etaa, \gamma, a)=b_k^p(\gamma, a).\end{equation}

For each $\alpha$ and each $W\in S(\tua)$, there exists $N\in S(\O)$ such that $\pi_\a$ maps $W$ isometrically onto $N\cap U_\a$.  Let
\begin{equation}\label{san}S(\tua;N)=\{W\in S(\tua): \pi_\a(W)=N\cap U_\a\} \end{equation}
 and observe that
\begin{equation}\label{osan}|S(\tua; N)|= \frac{|G_\a|}{|\Iso(N)|}.\end{equation}
For $W\in S(\tua;N)$, we identify $\isomax(N)$ with $\isomax(W)$.   Observe that
\begin{equation}\label{WN}\int_W\,c_k^p(\tetaa, \gamma, \tilde{a})\, dV_W(\tilde{a}) =\int _{N\cap U_\a}\,c_k^p(\etaa,\gamma, a)\, dV_N(a).\end{equation}
Since $\dim(W)=\dim(N)$, Equations~\eqref{ip}, \eqref{sumck}, \eqref{san}, \eqref{osan} and \eqref{WN}  yield
\begin{align*}
I'(t)&=\sum_{N\in S(\O)}\sum_{\a=1}^s \frac{|G_\a|}{|\Iso(N)|}\sum_{\g\in\isomax(N)}\,\frac{1}{|G_\a|}(4\pi t)^{-\dim(N)/2}\sum_{k=0}^\infty\,t^k\int _{N\cap U_\a}\,c_k^p(\etaa,\gamma, a) dV_N(a)\\
&=\sum_{N\in S(\O)}\frac{1}{|\Iso(N)|}\sum_{\g\in\isomax(N)}\,(4\pi t)^{-\dim(N)/2}\sum_{k=0}^\infty\,t^k\int_N\,b_k^p(\gamma, a) \, dV_N(a)=\sum_{N\in S(\O)} \, \frac{I_N^p(t)}{|\iso(N)|}.
\end{align*}
  The theorem follows.
\end{proof}

\begin{remark}\label{abs boun}
Let $\O$ be a closed Riemannian orbifold all of whose singular strata $N$ have co-dimension one; i.e., the singular strata are mirrors.   Every such stratum is necessarily totally geodesic in $\O$.  (Indeed, about any point $a$ in $N$, there exists a coordinate chart $(\tu, G_U,\pi_U)$ where $G_U$ is generated by a reflection $\tau$, and the Riemannian metric on $U=\pi(\tu)$ arises from a $\tau$-invariant metric on $\tu$.    The fixed point set of any involutive isometry is necessarily totally geodesic.)   The underlying topological space $\underline{\O}$ has the structure of a smooth Riemannian manifold each of whose boundary components is totally geodesic.

The orbifold $\O$ is good in this case, as we may double $\O$ over its reflectors
to obtain a smooth Riemannian manifold $M$ admitting a reflection symmetry $\tau$ so that $\O=\langle\tau\rangle\backslash M$.   Let $P:M\to \underline{\O}$ be the projection.  Identify the fixed point set of $\tau$ with $N$.  Let $\omega$ be a smooth $p$-form on the orbifold $\O$.  Then $\widetilde{\omega}:=P^*\omega$ is $\tau$ invariant.   Let $\nu$ be a unit normal vector field along $N$ in $M$ and let $j:N\to M$ be the inclusion.   The $\tau$-invariance of $\widetilde{\omega}$ is equivalent to the conditions $j^*\iota_\nu\widetilde{\omega}=j^*\iota_\nu d\widetilde{\omega}=0$.   Using the same notation $\nu$ for the unit normal $P_*\nu$ along $N$ in $\O$ and now letting $j:N\to \O$ be the inclusion, this gives
\[
j^*\iota_\nu \omega=j^*\iota_\nu d\omega=0,
\]
and these are precisely the absolute boundary conditions for the manifold $\underline{\O}$.  The spectrum of the Hodge Laplacian for $p$-forms on $\O$ thus coincides with the spectrum of the Hodge Laplacian for $p$-forms on $\underline{\O}$ with absolute boundary conditions. It is then straightforward to prove that the spectral invariants $b^p_k(N)$ computed here for the orbifold agree with the familiar contributions to the heat trace asymptotics arising from the boundary of $\underline{\O}$ as obtained in \cite[(2) Theorem 3.2]{P03} (see also \cite[Theorem 1.2]{BG90}).

\end{remark}

\section{Proof of Theorem~\ref{thm:main1V4}}\label{applications}

In this section we use the asymptotic results derived in the previous section to prove our main theorem.
For various types of singular strata $N$ in a $d$-dimensional closed orbifold $\orb$, we first compute the invariant $b_0^p(\gamma)$ for $\gamma \in \isomax(N)$.


\begin{prop}\label{b01general}  
Let $N$ be a singular stratum of codimension $k$ in the $d$-dimensional closed orbifold $\O$. Suppose $\gamma \in\isomax(N)$ has eigenvalue type $E(\theta_1, \theta_2, \dots, \theta_s;r)$, using Notation~\ref{nota:Rstuff}.

Then
\begin{equation}\label{eq.b01}b_0^1(\gamma)=\left( d-k-r+\sum_{j=1}^s\,2\cos(\theta_j)\right) \left( 2^{-k}\prod_{j=1}^s\,\csc^2(\theta_j/2)\right) .\end{equation}
Here we use the convention that when $s=0$, we have $\prod_{j=1}^s\,\csc^2(\theta_j/2)=1$.
\end{prop}

\begin{proof} We apply Theorem~\ref{bkthm}.
The expression in the first set of large parentheses in Equation~(\ref{eq.b01}) is the trace of $\gamma$.    The expression in the second set of large parentheses equals $\displaystyle\frac{1}{|\det(\Id_{\codim(N)}-A)|}$, where (taking note of Notation~\ref{invariants} and Notation and Remarks \ref{nota:don}), we are writing $A$ to denote the expression $A_a$ in Theorem~\ref{bkthm} part~\ref{it:bkthm-global}.
\end{proof}

In preparation for the proof of Theorem \ref{thm:main1V4}, we consider primary singular strata of codimension two with cyclic isotropy groups.

\begin{prop}
\label{b01codim2}
Let $N$ be a stratum of codimension two with cyclic isotropy group of order $m \geq 2$.  Then,
\begin{equation}\label{eq.codim2} 
b_0^1(N)= \left( (d-2)\frac{m^2 - 1}{12} + \frac{m^2 - 6m + 5}{6}\right) \vol(N).
\end{equation}
\end{prop}

\begin{proof} If $m=2$, then the generator $\gamma$ of $\Iso(N)$ must have eigenvalue type $E(;2)$.
If $m \geq 3$, $\gn$ is generated by an element $\gamma$ with eigenvalue type $E(2\pi/m;)$. 
All elements of $\gn$ other than the identity lie in $\isomax(N)$. The proposition then follows from Notation~\ref{invariants} part~\ref{it:invariants-IpNt}, Proposition~\ref{b01general}, and the following two formulas:
\[\sum_{j=1}^{m -1} \csc^2(\pi j/ m)
  = \frac{m^2 - 1}{3}
\ \ \ \
\text{and} \ \ \ \
\sum_{j=1}^{m -1} \cos (2\pi j/ m)\csc^2 (\pi j/ m)
  = \frac{m^2 - 6m + 5}{3},
\]
assuming here that $2\leq m \in \mathbb{Z}$. These formulas are proved using the calculus of residues.   See \cite[Lemma 5.4]{DGGW08} for the first formula.
For the second, by \cite[Equation (2.3)]{CS12}, we have
  \begin{align*}
    S_3(m,1,1) &= \sum_{j=1}^{m-1} \cos (2\pi j/ m) \csc^2 (\pi j/ m) \\
    &= 2 \sum_{\alpha =0}^{1} \binom{2}{2\alpha} B_{2\alpha}(1/m) B_{2-2\alpha}^{(2)}(1) m^{2\alpha} \\
    &= 2B_0(1/m)B_2^{(2)}(1) + 2B_2(1/m)B_0^{(2)}(1)m^2
    = \frac{m^2 - 6m + 5}{3},
  \end{align*}
where $B_{n}(x)$ are the ordinary Bernoulli polynomials and $B_{n}^{(m)}(x)$ are the higher-order Bernoulli polynomials in $x$ (see \cite{CS12} and references therein).

\end{proof}

\begin{ex}\label{compound rotation} Although not required for our main result, one can generalize Proposition~\ref{b01codim2} to compute $b_0^1(N)$ for strata for which $\gn$ is cyclic and
\begin{itemize}
\item when $m>2$, the generator has eigenvalue type $E(2\pi/m, \dots, 2\pi/m;)$ with $2\pi/m$ repeated $s$ times, or
\item when $m=2$, the generator has eigenvalue type $E(;r)$ with $r$ even.
\end{itemize}

Consider the case that $N$ has codimension $4$ in $\O$. By \cite[Equation~(5.2) and Corollary~5.4]{BY02}, we have
\begin{equation}\label{trigsum4}
    \sum_{j=1}^{m-1} \cos (2\pi j/ m) \csc^4(\pi j/ m) = \frac{m^4 -20m^2 +19}{45}.
\end{equation}

From \cite[Equation~(5.2) and Corollary~5.2]{BY02}, we have that
\begin{equation}\label{trigsum4'}
    \sum_{j=1}^{m -1} \frac{d-4}{\sin^4(\pi j/ m)}
    = \frac{(d-4) (m^4 +10m^2 -11)}{45}.
\end{equation}

Thus by Proposition~\ref{b01general},
\begin{equation*}
    b_0^1(N) =
   \vol(N) \frac{4(m^4 -20m^2 +19)+(d-4) (m^4 +10m^2 -11)}{720}.
\end{equation*}

Reference \cite{BY02} also contains formulas analogous to Equations~(\ref{trigsum4}) and (\ref{trigsum4'}) with the power $4$ replaced by the power $2s$ for arbitrary positive $s$, thus yielding expressions for $b_0^1(N)$ for rotation strata of constant rotation angle and higher codimension.

\end{ex}


\begin{proof}[Proof of Theorem~\ref{thm:main1V4}]
If the singular set of $\O$ has odd codimension, then Theorem~\ref{thm:dggw5.1} and the fact that strata of minumum codimension are primary imply that the 0-spectrum alone suffices to distinguish $\O$ from any Riemannian  manifold.   Thus we may assume that the singular set has codimension two.   Let $S_2(\O)$ denote the set of all strata of codimension two.

Let $M$ be a closed Riemannian manifold of dimension $d$.   In the notation of Equation~\eqref{heatascoeffc}, we have
\begin{equation}\label{c2}c_2^p(M)=a_1^p(M)\,\,\,\mbox{while}\,\,\,c_2^p(\O)= a_1^p(\O) +\sum_{N\in S_2(\O)}\,\frac{4\pi}{|\iiso(N)|}b_0^p(N).\end{equation}
For $* =M$ or $\O$, we have
\begin{equation}\label{eqc} a_1^1(*)=\frac{d-6}{6} \int_{*} \tau \, dV_*=(d-6)a_1^0(*).\end{equation}
We will show that
\begin{equation}\label{greater d-6}
b_0^1(N)>(d-6)b_0^0(N)
\end{equation}
 for every $N\in S_2(\O)$.    Theorem~\ref{thm:main1V4} will then follow from Theorem~\ref{thm.asympt} and Equations~\eqref{c2}, \eqref{eqc} and \eqref{greater d-6}.

 For $N\in S_2(\O)$, we have $\iso(N)\simeq A\times  I_{d-2}$ where $A\subset O(2)$ and $I_{d-2}$ is the $(d-2)\times (d-2)$ identity matrix.  Since $\O$ contains no strata of codimension one, $A$ cannot contain any reflections and thus $N$ has cyclic isotropy group of order $m$ contained in $SO(2)\times  I_{d-2}$.
The computation of $b_0^0(N)$ is analogous to the case of cone points in dimension 2, carried out in \cite[Proposition 5.5]{DGGW08}, yielding
$$b_ 0^0(N)=\frac{(m^2-1)}{12} \vol(N).$$
Using Proposition~\ref{b01codim2}, we have
$$
b_0^1(N)=\left((d-2)\frac{m^2-1}{12}+\frac{m^2-6m+5}{6}\right)\vol(N).
$$
Thus, noting that $m\geq 2$, we have
$$b_0^1(N)-(d-6)b_0^0(N)=\frac{3m^2 - 6m + 3}{6}>0.$$
This proves Equation~(\ref{greater d-6}) and the theorem follows.
\end{proof}

\begin{remark}\label{rem:orbisurfaces} It is shown in \cite[Theorem 5.15]{DGGW08} that the $0$-spectrum alone suffices to distinguish singular, closed, locally orientable Riemannian orbisurfaces with nonnegative Euler characteristic from smooth, oriented, closed Riemannian surfaces. In fact, it is shown there that the degree zero term in the small-time heat trace asymptotics for functions gives rise to a complete topological invariant for orbisurfaces satisfying these constraints. In Theorem~\ref{thm:main1V4}, by also appealing to the $1$-spectrum, we do not require local orientability or nonnegative Euler characteristic to distinguish Riemannian orbisurfaces from closed Riemannian surfaces.
\end{remark}

\section{Appendix}
\label{app}

We outline the proof of Theorem~\ref{bkthm}.    We are closely following the proof of \cite[Theorem 4.1]{D76} but expressing it in a more general context.

Let $M$ be an arbitrary Riemannian manifold and $\gamma$ an isometry of $M$. Let $d(x,y)$ denote the induced distance function between $x$ and $y$ in $M$.  Outside of a small tubular neighborhood of $\Fix(\gamma)$, we have that  $(4\pi t) ^{-d/2}e^{-\frac{ d^2(x, \gamma(x))}{4t}}\to 0$ uniformly as $t\to 0^+$.  Thus for any compactly supported smooth function $f$ on $M$, we have
\begin{align*}
&\int_{M}  \,(4\pi t) ^{-d/2}e^{-\frac{ d^2(x, \gamma(x))}{4t}}f(x) \,dV_{ M}(x)
\\
&\qquad = \sum_{Q\in \mathcal{C}(\Fix(\gamma))} \int_{U_Q} \,(4\pi t) ^{-d/2}e^{-\frac{ d^2(x, \gamma(x))}{4t}}f(x) \,dV_{ M}(x)
+ O(e^{-c/t}),
\end{align*}
where $U_Q$ is a small tubular neighborhood of $Q$, say of radius $r$, and $c>0$ is a constant, where, as in Notation and Remarks \ref{nota:don} part~\ref{it:don-CFix}, $\mathcal C(\Fix(\gamma))$ denotes the set of connected components of $\Fix(\gamma).$

We will later specialize to the choices of $f$ that arise in Theorem~\ref{bkthm}, but for now we work in the general setting.

As in \cite[Theorem 4.1]{D76}, we analyze each term in the sum above individually and define
\begin{equation*}
I_Q(f):= \int_{U_Q} \,(4\pi t) ^{-d/2}e^{-\frac{ d^2(x, \gamma(x))}{4t}}f(x) \,dV_{ M}(x).
\end{equation*}

 Let $W_Q\subset TM$ be the normal disk bundle of $Q$ of radius $r$ and let $\varphi: W_Q\to Q$ be the bundle projection.   We use the diffeomorphism $W_Q\to U_Q$ given by $\x\mapsto\exp_{\varphi(\x)}\x$, where $\exp$ is the Riemannian exponential map of $M$, to express $I_Q(f)$ as
$$
I_Q(f)= \int_Q\int_{\varphi^{-1}(a)} (4\pi t) ^{-d/2} e^{-\frac{ d^2(x, \gamma(x))}{4t}}f(x)\,\psi(\x) \,d\x \,dV_{Q}(a).
$$

Here $d\x$ is the Euclidean volume element defined by the Riemannian structure on $(T_aQ)^\perp$, and for $\x\in \varphi^{-1}(a)$, we are writing $x=\exp_a(\x)$.    The function $\psi$ arising in this change of variables depends only on the curvature and its covariant derivatives and, as shown in \cite[Equation~(2.6)]{D76}, satisfies
\begin{equation}\label{psi}\psi(\x)=1-\frac{1}{2}R_{i\alpha j\alpha}\x^i\x^j-\frac{1}{6}R_{ikjk}\x^i\x^j+O(\x^3), \end{equation}
where the $\x^i$ are the coordinates of $\x$ with respect to an orthonormal basis of $(T_{a}Q)^\perp$ and where the components of the curvature tensor are evaluated at the point $a$.

Letting $\bar{\x}=\x-d\gamma_a(\x)$, Donnelly shows that $d^2(x, \gamma(x))= (\bar{\x})^2 +O(\bar{\x})^3$ and then applies a version of the Morse Lemma, see \cite[Lemma A.1]{D76}, to find new coordinates $\y$ so that
\begin{equation*}
d^2(x, \gamma(x))=\sum_{i=1}^s\y^2_i=\|\y\|^2,
\end{equation*}
Making the two changes of variables $\x\to \bar{\x}\to \y$ on $(T_aQ)^\perp$, we have
$$d\x =|\det(B(a))|\,d\bar{\x}=|\det(B(a))||J(\bar{\x},\y)|d\y,$$
where
\begin{equation}\label{B} B(a)= (\Id_{\codim(Q)}-A_a)^{-1}.\end{equation}
 Here $A_a$ is the restriction of $d\gamma_a$ to $(T_aQ)^\perp$ and $|J(\bar{\x},\y)|$ denotes the absolute value of the Jacobian determinant for the second change of variables.
 We are viewing $x=\exp(\x)$, $\x$ and $\bar{\x}$ as functions of the new variable $\y$.    Donnelly shows that $|J(\bar{\x},\y)|$ has a Taylor expansion in $\y$ whose coefficients depend only on the values at $a$ of $B$ and of the curvature tensor $R$ of $M$ and its covariant derivatives:
 \begin{equation}\label{J}
 |J(\bar{\x},\y)|=1+\frac{1}{6}(R_{ikih}B_{ks}B_{ht} +R_{iksh}B_{ki}B_{ht}+R_{ikth}B_{ks}B_{hi})\y^s\y^t +O(\y^3).
 \end{equation}

Letting $H_a(f):(T_aQ)^\perp\to \R$ be given by
\[
H_a(f)(\y)\coloneqq f(x) |\det(B(a))|\,| J(\overline{\x}, \y)| \,\psi(\x)
\]

we have
\begin{equation}\label{iqf} I_Q(f)=\int_Q\,\int_{\varphi^{-1}(a)}\,(4\pi t)^{-d/2}e^{-\|\y\|^2/4t}\,H_a(f)(\y)\,d\y\,dV_{Q}(a).\end{equation}

We expand $H_a(f)$ into its Maclaurin series.  Because of symmetry, only the terms of even degree contribute to the integral $I_Q(f)$.   Thus denoting by $[H_a(f)]_j(\y)$ the homogeneous component of degree $2j$ in the Maclaurin series of $H_a(f)$ and then making the change of variable $\z=\frac{\y}{\sqrt{t}}$, we have
\begin{multline*}\label{expand} I_Q(f)=(4\pi t)^{-\dim(Q)/2}\int_Q\,\int_{\R^{\codim(Q)}}\,(4\pi )^{-(\codim(Q))/2}e^{-\|\z\|^2/4}\sum_{k=0}^L\,t^k [H_a(f)]_k(\z)\,d\z\,dV_Q(a) \\+O(t^{L+1}).
\end{multline*}
(Here we are using the fact that $\varphi^{-1}(a)$ is a disk of radius $\frac{r}{\sqrt{t}}$ in $\z$, which can be replaced by $\R^{\codim(Q)}$ without changing the asymptotics.)

Proceeding exactly as in the proof of \cite[Theorem 4.1]{D76}, we complete the integration by doing an iterated integral over the $\z$ variables, noting that odd powers of any of the coordinates $\z_i$ lead to integrals of the form $\int_\mathbb{R}\,\z_i^{2k+1}e^{-(\z_i^2)/4}\,d\z_i$ which are equal to zero, while the contributions from the terms of the form $\int_\mathbb{R}\z_i^{2m}e^{-(\z_i^2)/4}\,d\z_i$ can be computed by using the classical formula
\begin{equation*}
\int_{\mathbb R}x^{2m}e^{-x^2}\,dx=\frac{1\cdot3\cdot5\cdots (2m-1)}{2^{m+1}}\sqrt{\pi}.
\end{equation*}
We are left with
\begin{equation*}
I_Q(f)=(4\pi t)^{-\dim(Q)/2}\sum_{k=0}^L\frac{t^k}{k!}\int_Q\, \Box^k_{\y}(H_a(f))(0)\,dV_Q(a) +O(t^{L+1}),
\end{equation*}
where, following the notation of Donnelly in \cite[Theorem 4.1]{D76}, we denote
\[
\Box_\y:=\sum_{i=1}^s\frac{\partial^2}{\partial \y_i^2}.
\]

Thus we conclude:

\begin{prop}\label{f}  For any compactly supported smooth function $f$ on $M$, we have
\begin{multline*}
\int_{M}  \,(4\pi t) ^{-d/2}e^{-\frac{ d^2(x, \gamma(x))}{4t}}f(x) \,dV_{ M}(x)
\\
= \sum_{Q\in \mathcal{C}(\Fix(\gamma))} (4\pi t)^{-\dim(Q)/2}\sum_{k=0}^L\frac{t^k}{k!}\int_Q\, \Box^k_{\y}\left(H_a(f)\right)(0)\,dV_Q(a) +O(t^{L+1})\end{multline*}
where $H_a(f)=\frac{f(x)\,| J(\overline{\x}, \y)| \,\psi(\x)}{|\det(\Id_{\codim(Q)}-A_a)|}$ with $\psi(x)$ and $| J(\overline{\x}, \y)|$ given as in Equations~\eqref{psi} and \eqref{J} and with $A_a$ being the restriction of $d\gamma_a$ to $(T_aQ)^\perp$.
\end{prop}
\medskip

To complete the proof of Theorem~\ref{bkthm} part~\ref{it:bkthm-global}, write
\begin{equation}\label{fjp}f_j^p(x)=C_{1,2} \gamma^*_{\cdot, 2} u^p_j(x,x).\end{equation}
Then

\begin{multline*}\int_{M} C_{1,2} \gamma^*_{\cdot, 2} K^p(t,x,x) dV_{M}(x) = \sum_{j=0}^L\, t^j\int_{M}\,(4\pi t) ^{-d/2} e^{-\frac{ d^2(x, \gamma(x))}{4t}}f_j^p(x)dV_{M}(x)+O(t^{L+1})\\
= \sum_{Q\in \mathcal{C}(\Fix(\gamma))}\,(4\pi t)^{-\dim(Q)/2}\sum_{k=0}^{L} t^k\,\int_{Q} b^p_k(\gamma,a) dV_Q (a) +O(t^{L+1}),
\end{multline*}
where
$$b^p_k(\gamma,a)=\sum_{j=0}^k\,\frac{1}{j!}\Box^j_{\y}\left(H_a(f^p_{k-j})\right)(0)$$
with $f^p_{k-j}$ defined as in Equation~\eqref{fjp}.

In particular, when $k=0$, we have
\begin{equation}\label{bop}
b^p_0(\gamma,a)= H_a(f_0^p)(0)=\frac{C_{1,2}\gamma^*_{\cdot, 2} u^p_0(a,a)}{\lvert \det(\Id_{\codim(Q)}-A_a)\rvert} = \frac{\tr_p(d\gamma_a)}{\lvert \det(\Id_{\codim(Q)}-A_a)\rvert},
\end{equation}
where $\tr_p(d\gamma_a)$ is defined as in Notation and Remarks \ref{nota:don} part~\ref{it:don-trace}.  (The second equality follows from the fact that $u_0^p(a,a)$ is the identity transformation of $\Lambda^p (T^*_a M)$
as noted in Proposition~\ref{localparam} part~\ref{uzero}.)

Theorem~\ref{bkthm} part~\ref{it:bkthmlocal} also follows with
$$c^p_k(\eta,\gamma,a)=\sum_{j=0}^k\,\frac{1}{j!}\Box^j_{\y}\left(H_a(\eta f^p_{k-j})\right)(0).$$

\textbf{Funding}: This work was supported by the Banff International Research Station and Casa Matem\'atica Oaxaca [19w5115]; the Association for Women in Mathematics National Science Foundation ADVANCE grant [1500481]; and the Swiss Mathematical Society [SMS Travel Grant to KG]. IMS was supported by the Leverhulme Trust grant [RPG-2019-055].

\textbf{Acknowledgments}:
We thank Carla Farsi for a helpful discussion of orbifold orientability.  We also thank the Banff International Research Station and Casa Matem\'atica Oaxaca for initiating our collaboration by hosting the 2019 Women in Geometry Workshop. In addition, we thank the Association for Women in Mathematics and the organizers of the Women In Geometry Workshop for supporting this collaboration. IMS thanks Prof.~Jacek Brodzki for providing her with excellent conditions to work.

\bibliographystyle{alpha}
\bibliography{colibrisbib}

\end{document}